\documentclass[11pt]{amsart}
\usepackage[usenames,dvipsnames]{pstricks}
\usepackage[mathscr]{eucal}
\usepackage{amsfonts,amsmath,amssymb,amsthm,amscd
	,amsxtra,stmaryrd}
\usepackage[all,2cell,ps]{xy}
\usepackage{tikz}
\usepackage{lscape}
\usepackage{stmaryrd}
\usetikzlibrary{matrix,arrows,decorations.pathmorphing}
\usepackage[notcite, notref]{}
\usepackage[pagebackref]{hyperref}
\usepackage{calc,color}
\usepackage{wasysym}
\usepackage{mathptmx}
\usepackage{amsmath}
\usepackage{amssymb}
\usepackage{amsfonts}
\usepackage{amsthm}
\usepackage{enumerate}
\usepackage{color}
\usepackage{amscd}
\usepackage{calc}

\newtheorem{thm}{Theorem}[section]
\newtheorem{cor}[thm]{Corollary}
\newtheorem{lem}[thm]{Lemma}
\newtheorem{prop}[thm]{Proposition}
\newtheorem{defn}[thm]{Definition}

\newtheorem{rem}[thm]{Remark}

\newtheorem{Notation}[thm]{Notation}

\newcommand{\Hom}{\mbox{Hom}\,}
\newcommand{\Ext}{\mbox{Ext}\,}
\newcommand{\Tor}{\mbox{Tor}\,}

\newcommand{\depth}{\mbox{depth}\,}
\renewcommand{\dim}{\mbox{dim}\,}

\newcommand{\Min}{\mbox{Min}\,}
\newcommand{\pd}{\mbox{proj.dim}\,}

\newcommand{\id}{\mbox{inj.dim}\,}

\newcommand{\gd}{\mbox{G-dim}\,}

\newcommand{\h}{\mbox{ht}\,}
\newcommand{\im}{\mbox{Im}\,}
\newcommand{\E}{\mbox{E}}

\renewcommand{\H}{\mbox{H}}

\newcommand{\F}{\mbox{F}}

\newcommand{\D}{\mbox{D}}

\newcommand{\fm}{\mathfrak{m}}
\newcommand{\fp}{\mathfrak{p}}
\newcommand{\fP}{\mathfrak{P}}
\newcommand{\fq}{\mathfrak{q}}
\newcommand{\fQ}{\mathfrak{Q}}
\newcommand{\fn}{\mathfrak{n}}

\begin{document}

\title{ On the interplay between the Frobenius functor and its dual}

\author[Dibaei]{Mohammad T. Dibaei$^1$}
\address{$^1$ Faculty of Mathematical Sciences and Computer, Kharazmi University, Tehran, Iran; and School of Mathematics, Institute for Research in Fundamental Sciences (IPM), P.O. Box: 19395-5746, Tehran, Iran.}
\email{dibaeimt@ipm.ir}

\author[Eghbali]{Mohammad Eghbali$^2$}
\address{$^2$ Faculty of Mathematical Sciences and Computer, Kharazmi University, Tehran, Iran} \email{m.eghbali007@yahoo.com}

\author[Khalatpour]{Yaser Khalatpour$^3$}
\address{$^3$ Faculty of Mathematical Sciences and Computer, Kharazmi University, Tehran, Iran} \email{yaserkhalatpour@gmail.com}


\begin{abstract}
For a commutative Noetherian ring $R$ of prime characteristic, denote by ${^f\hspace{-0.5mm}R}$  the ring $R$ with the left structure given by the Frobenius map. We develop Thomas Marley's work on the property of the Frobenius functor $\F(-) = - \otimes_R {^f\hspace{-0.5mm}R}$ and show some interplays between $\F$ and its dual   $\widetilde{\F}(-) = \Hom_R({^f\hspace{-0.5mm}R}, -)$ which is introduced by J\"{u}rgen Herzog.
\end{abstract}


\maketitle


\section{Introduction}
Let $R$ be a commutative Noetherian ring of prime characteristic $p$ and let $f$ be the Frobenius map. Consider ${^f\hspace{-0.5mm}R}$ which, as an additive group coincides with $R$, and has bimodule structure over $R$ such that $R$ acts on the left via $f$ and on the right via the identity. The functor $\F(-) = - \otimes_R {^f\hspace{-0.5mm}R}$, which has been introduced by  Peskine-Szpiro, plays an important role in commutative algebra (e.g., see \cite{Kun69}, \cite{PS73},\cite{HE74} and \cite{Kun76}).
Note that $\F$ preserves projectivity. Thomas Marley, in  \cite{MA14}, has studied rings on which $\F$ preserves injectivity and characterized one dimensional rings for which $\F$ preserves injective modules ($\F$PI for short). 

Our first task is to develop the Marley's idea to rings of higher dimensions in Section 3, Propositions \ref{prop2.11}.
By using the functor $\widetilde{\F}(-)=\Hom_R({^f\hspace{-0.5mm}R}, -)$, we show that a necessary and sufficient condition for a ring to be $\F$PI is that the two functors $(\F(-))^{*}$ and  $\widetilde{\F}((-)^{*})$ are equivalent on the category of finitely generated $R$-modules, where $(-)^{*} = \Hom_R (-, R)$ (see Theorem \ref{thmFF}).  As an application, we next prove that, under mild conditions, $R$ is quasi-Gorenstein if and only if $R$ is $\F$PI (Theorem\ref{thm3}).
Finally, we introduce rings for which $\widetilde{\F}$ preserves reflexivity ( $\widetilde{\F}$PR for short),  and show that an $\F$-finite one-dimensional local ring is Gorenstein ( resp. $\F$PI) if and only if it is $\widetilde{\F}$PR and $\id_R \widetilde{\F}(R) < \infty $  (resp. $R$ is Cohen-Macaulay and $\widetilde{\F}(R)$ has a non-trivial free direct summand) (see Theorem \ref{FtildeF}).

\section{Preliminaries}
Throughout this paper $R$ is a commutative Noetherian ring with prime characteristic $p$ and all modules are finitely generated $R$-modules unless otherwise stated explicitly. The Frobenius map $f:R\rightarrow R$ is defined by $f(r)=r^p$ for all $r\in R$ which is a ring homomorphism. Let ${^f\hspace{-0.5mm}R}$ denote the $(R-R)$-bimodule  which is $R$ as an additive group and has the structure on the left defined by $f$ and the structure on the right defined by the identity  map on $R$. For a positive integer $e$, denote $f^e=f\circ\cdots\circ f: R\rightarrow R$, $e$ times. Throughout the paper, we use some key functors derived from $f^e$. 

\begin{Notation} \label{l1}
	\begin{enumerate}[\rm(i)]
		\item \emph{Let ${^{f^e}}\hspace{-0.5mm}(-)$ be the functor from the category of $R$-modules to itself such that, for an $R$-module $M$, $ ^{f^e}\hspace{-0.5mm} M$ denote the abelian group $M$ viewed as a $(R-R)$-bimodule  via  the left and right scaler products $r\ast m = r^{p^e}m$ and $m\cdot r=rm$, $r\in R$, $m \in M$, respectively.
The functor ${^{f^e}}\hspace{-0.5mm}(-)$ is exact on the category of $R$-modules.}

\item \emph{The $e$th Peskine-Szpiro functor is $ \F^e _R (-) := (-) \otimes_R {^{f^e}\hspace{-0.5mm}R}$. For an $R$-module $M$, $\F^e_R(M)$ has been given an $R$-module structure as follows. For $r \in R$, $s\in {^{f^e}\hspace{-0.5mm}R}$ and $x\in M$, $r. (x \otimes s)= x \otimes (sr)$. Note that $(rx) \otimes s = x \otimes (r^{p^e}s) $. When the ring $R$ is clear in the context, we write $\F^e(M)$ instead of $\F^e_R(M)$ and for $e = 1$ write $\F (M)$.
Note that  $\F^e$ is an additive and right exact functor which preserves direct sums and direct limits. It is routine to check that  $\F (R) \cong R$. 
It follows that, for any ideal $I$ of R, $\F(R/I) \cong R/I^{[p]} $, where $I^{[p]} = (a^p | a \in I) $. Iterating, we obtain $\F^e(R)\cong R$ and $\F^e (R/I) \cong R/I^{[q]}$, where $q=p^e$.}

\item \emph{As the dual of the Peskine-Szpiro functor $\F^e$, recall the functor $\widetilde{\F}_R ^e (-) = \Hom_R ({^{f^e}\hspace{-0.5mm}R} , -) $ on the category of $R$-modules from \cite[\S2, 2]{HE74}. For an $R$-module $M$,  $\widetilde{\F}_R ^e (M)$ is an abelian group which has an $R$-module structure via the right $R$-action on ${^{f^e}\hspace{-0.5mm}R}$, i.e. for any  $r\in R$, $\theta\in \Hom_R ({^{f^e}\hspace{-0.5mm}R} , M)$ and $s\in {^{f^e}\hspace{-0.5mm}R}$,  $(\theta r)(s) = \theta (sr)$. As $\theta$ is an $R$-homomorphism of left $R$-modules, one can see that $ \theta (r^{p^{e}}s) = r\theta(s)$. Note that  $\widetilde{\F}_R ^e (-)$ is an additive and left exact functor on the category of $R$-modules which preserves  inverse limits. We use symbol $\widetilde{\F}^e$ instead of  $\widetilde{\F}^e_R$  and $\widetilde{\F}$ for the case $e=1$. For more details about ${^{f^e}}\hspace{-0.5mm}(-)$, $\F^e$ and $\widetilde{\F}^e$ see \cite{Kun69}, \cite{PS73}, \cite{HE74}, \cite{Go77}, \cite{BH98}.}
\end{enumerate}
\end{Notation}
We begin by the following well-known definition and remark.
\begin{defn}
\emph{ A ring $R$ is said to be $\F$-{\it finite} if ${^f\hspace{-0.5mm}R}$ is a finitely generated
left $R$-module.}
\end{defn}
The ring $R$ is $\F$-finite if and only if ${^{f^e}\hspace{-0.5mm}R}$ is a finitely generated $R$-module for any $e\geqslant 1$. For any prime ideal $\fq$ of $R$, $({^f\hspace{-0.5mm}R})_{\fq}\cong {^f\hspace{-0.5mm}(R_{\fq})}$ as $R_{\fq}$-module, without $R$ being $\F$-finite. The property of being $\F$-finite is preserved under localizations and finitely generated algebra extensions. Also, when $(R, \fm)$ is local, one has ${^f\hspace{-0.5mm}(\widehat{R})}\cong \widehat{^f\hspace{-0.5mm}R}$ as $\widehat{R}$-module. Therefore, if $R$ is $\F$-finite then $\widehat{R}$ is $\F$-finite.

\begin{rem} \label{rem2.1}
\emph{ Let $(R, \fm, k)$ be a local ring.
\begin{itemize}
\item Let $R$ be an Artinian local ring, $M$ a finitely generated $R$-module. Consider a minimal presentation $R^n \overset{(a_{ij})}{\longrightarrow} R^m \rightarrow M \rightarrow 0$ for $M$. Applying $\F^e$, gives the exact sequence $R^n \overset{(a_{ij}^{p^e})}{\longrightarrow} R^m \rightarrow \F^e(M) \rightarrow 0$. As $(a_{ij}^{p^e}) = 0$ for sufficiently large $e$, one obtains $\F^e(M) \cong R^m$ for $e\gg 0$.
\item When $R$ is $\F$-finite, $[k:k^p]$ is a power of $p$. Set $\alpha = \log_p [k:k^p]$.
For any ideal $I$ of $R$ and any $R$-module $M$, one has $ \frac{R}{I} \otimes_R {^{f^e}\hspace{-0.5mm}M} \cong {^{f^e}\hspace{-0.5mm}(M/I^{[q]}M)}$ as left $R$-modules, where $q=p^e$ (see Notation \ref{l1}(i)). If $M$ has finite length, i.e. $ \ell_R(M)< \infty  $, one may see that $\ell_R ({^{f^e}\hspace{-0.5mm}M}) = q ^ {\alpha} \ell_R (M)$ for every $e$ and ${^{f^e}\hspace{-0.5mm}M} \cong \overset{t}{\oplus} k$ as left $R$-module for all $e \gg 0$, where $t= q ^ {\alpha} \ell_R (M)$.
\end{itemize}}
\end{rem}

For a local ring $(R, \fm, k)$, set $\D_R (-)= \Hom_R (- , \E_R(R/{\fm}))$ to be  the Matlis dual functor. We use the result due to Herzog which shows the interplay between  $\F$ and $\widetilde{\F}$.

\begin{prop} \cite[Proposition 4.2]{HE74}.  \label{DFtildF}
Let $(R, \fm)$ be an $\emph\F$-finite local ring.   Then, for any $e>0$, the following statements hold true.
\begin{enumerate} [\rm(i)]
\item $\emph\D_R (\widetilde{\emph\F}^e (A)) \cong \emph\F^e (\emph\D_R (A))$ for any Artinian $R$-module $A$.
\item $\widetilde{\emph\F}^e (\emph\D_R (M)) \cong \emph\D_R (\emph\F^e (\widehat{M}))$ for any finitely generated $R$-module $M$.
\end{enumerate}
\end{prop}

\section{Rings whose Frobenius functors  preserve injective modules}
In \cite{MA14}, Thomas Marley has characterized rings whose Frobenius functor $\F$ preserves injective modules and called such rings $\F$PI  rings. In other words, $R$ is $\F$PI if $\F(I)$ is injective for every injective $R$-module $I$. The ring $R$ is called {\it weakly} $\F$PI if $\F (I)$ is injective for every  injective Artinian $R$-module $I$ (see \cite[Definition 3.1]{MA14}).  Note that if $R$ is a homomorphic image of a Gorenstein ring then $\F$PI and {\it weakly} $\F$PI are equivalent (see \cite[Theorem 3.11]{MA14}). 
We use the following result frequently. 
\begin{prop}\label{m1} \cite[Proposition 3.4 (b)]{MA14} Every $\F$PI ring is generically Gorenstein.
\end{prop}  
Note that the converse of the above proposition may not be true. Set $R = k[[x,y]]/(xy,y^2) $. As $\depth R=0$, $R$ is not Cohen-Macaulay and so it is not $\F$PI, by \cite[Proposition 3.12]{MA14}. Note that $R_{(y)}$ is a field and so $R$ is generically Gorenstein.	
In \cite[Theorem 4.1]{MA14}, it is shown that the converse is also true under mild conditions in one dimensional case.
In the sequel, we  concentrate on Marley's study of $\F$PI rings. Before that, we recall some properties of a canonical module $K_R$,  if it exists, of a ring $R$ which is not necessarily Cohen-Macaulay. 

Let $(R, \fm)$ be a local ring of dimension $d$. A finitely generated $R$-module $K_R$ is called a canonical module of $R$ if $K_R \otimes_R\widehat{R}\cong\Hom_R(\H_{\fm}^d(R), \E_R(R/\fm))$, where $\H_\fm^d(R)$ is the top local cohomology module of $R$ with respect to the maximal ideal $\fm$ of $R$.

If a ring $R$ is a homomorphic image of a Cohen-Macaulay local ring $(S, \fn)$ possessing a canonical module $\omega_S$, then $R$ admits a canonical module $K_R\cong\Ext_S^{\dim S-\dim R}(R, \omega_S)$. In particular, if $(R, \fm)$ is a characteristic $p$ local ring which is $\F$-finite, then $R$ is a homomorphic image of an $\F$-finite regular local ring \cite[Remark 13.6]{Ga04}  and so  $R$ admits a canonical module $K_R$.
Note that when $K_R$ exists, the natural homomorphism $R \rightarrow \Hom_R(K_R, K_R)$ is an isomorphism if and only if $R$ is $(S_2)$ which is equivalent to that $\widehat{R}$ satisfies $(S_2)$.

 In \cite[Theorem 4.1]{MA14}, Marley has proved that a one-dimensional local ring $R$ satisfies $\F$PI if and only if $R$ is Cohen-Macaulay and admits a canonical ideal $\omega_R$ such that $\omega_R\cong\omega_R^{[p]}$.   In the following we develop this result for more general higher dimensions.	

\begin{thm}\label{prop2.11}\label{1} \emph{(Compare with \cite[Theorem 4.1]{MA14})}
Assume that $(R, \fm)$ is a local ring
satisfying the Serre condition $(S_2)$ 
which admits a canonical module $K_R$. 
\begin{enumerate}[\rm(i)]
\item If $R$ is generically Gorenstein and $K_R \cong K_R ^{[p]}$, then $R$ is $\emph\F$PI. 
\item If $R$ is an $\emph\F$PI Cohen-Macaulay  local ring such that $R / {K_R}^{[p]}$ is unmixed,  then $K_R \cong K_R^{[p]}$.
\end{enumerate}
\end{thm}

\begin{proof}
(i). As $R$ is equidimensional and unmixed, i.e. $R$ has no embedded primes, the canonical module $K_R$ is isomorphic to an ideal $I$ of $R$ such that $\h (I) = 1$ (see \cite[Proposition 2.4, Proposition 2.6]{LMa14}.
Note that, by \cite[Corollary 4.3]{Aoy83}, all formal fibres of $R$ are generically Gorenstein which implies that $\widehat{R}$ is also generically Gorenstein. By using \cite[Proposition 3.4 (d)]{MA14}, in order to prove $R$ is $\F$PI, we may assume that $R$ is complete. Moreover, we will finish the task,  by \cite[Theorem 3.11]{MA14}, if we prove $\F(\E) \cong \E$ where $\E=\E_R(R/\fm)$.

Consider the natural surjection map $\pi : \F(I) \rightarrow I^{[p]}$, given by $\pi (x \otimes s) = s x^p$, $x\in I$, $s\in {^f\hspace{-0.5mm}R}$. Let $\fp$ be a minimal prime ideal of $R$.  As $R$ is generically Gorenstein and $\h (I)= 1$, we have $\F(I)_{\fp} \cong R_{\fp}$, $I^{[p]} R_{\fp} = R_{\fp}$ and so $(\ker\pi)_{\fp} = 0$, i.e. $\dim_R(\ker\pi) < d$, where  $d=\dim R$. 
Applying local cohomology to the exact sequence $0\to\ker\pi\to\F(I)\to I^{[p]}\to 0$ yields the isomorphism $\H_{\fm} ^d (\F(I)) \cong \H_{\fm}^d (I^{[p]})$. Now we have
\[\begin{array}{rlll}
\F(\E) & \cong\F(\H_{\fm}^d (I)) & \text{( by \cite[12.1.20 Theorem]{BS12})} \\
&\cong \H_{\fm} ^d (\F(I)) & \text{( by \cite[Proposition 2.1]{MA14})} \\
&\cong \H_{\fm} ^d (I^{[p]}) \\
&\cong \H_{\fm} ^d (I) & \text{( by assumption)}\\
& \cong \E,
\end{array}
\] as desired.

(ii). For convenience set $\omega := K_R$. As $R$ is $\F$PI, it is generically Gorenstein (see Proposition \ref{m1}) so that  $R/\omega$  is Gorenstein and $R/ \omega^{[p]}$ is equidimensional.  Note that $R/{\omega^{[p]}}$ is a homomorphic image of $R$ so that it admits a canonical module, $K_{R/{\omega^{[p]}}}$ say. Also, by assumption, $R/{\omega^{[p]}}$ is unmixed  so that the natural map 
$R/{\omega^{[p]}} \to \Hom_{R/{\omega^{[p]}}} (K_{R/{\omega^{[p]}}}, K_{R/{\omega^{[p]}}})$ is injective, by \cite[Definition (2.1) and Remark (2.2)]{HH94}. Applying $-\otimes_{R/\omega^{[p]}} \widehat{R/\omega^{[p]}}$, we obtain the injective natural map 
$\widehat{R/\omega^{[p]}} \to \Hom_{\widehat{R/\omega^{[p]}}} (K_{\widehat{R/\omega^{[p]}}}, K_{\widehat{R/\omega^{[p]}}})$.
Therefore  $\widehat{R/\omega^{[p]}}$ is unmixed. On the other hand,
$R$ is a homomorphic image of a Gorenstein ring so that $\F$PI and  weakly $\F$PI are equivalent which is also equivalent to say that $\widehat{R}$ is weakly $\F$PI \cite[Proposition 3.4 (c)]{MA14}. Thus we may assume that $R$ is complete.

Applying local cohomology to the exact sequence
$ 0 \rightarrow \omega \rightarrow R \rightarrow {R/\omega} \rightarrow 0 $ gives the exact sequence $ 0 \rightarrow \H_{\fm}^{d-1} (R / \omega) \rightarrow \E \rightarrow \H_{\fm}^d (R) \rightarrow  0$, from which we obtain the exact sequence $ \Tor_1^R (\H_{\fm} ^d(R) , {^f\hspace{-0.5mm}R})  \rightarrow \F(\H_{\fm}^{d-1} (R/\omega^{[p]})) \rightarrow \F(\E) \rightarrow \F(\H_{\fm}^d (R)) \rightarrow  0.$ 

Note that, by \cite[Lemma 2.2.]{MA14}, $\Tor_1^R (\H_{\fm} ^d(R) , {^f\hspace{-0.5mm}R}) = 0$. Also $\F(\H_{\fm}^{d-1} (R / \omega)) \cong \H_{\fm}^{d-1} (R / \omega^{[p]})$, $\F(\E) \cong \E$ and $\F(\H_{\fm}^{d} (R)) \cong \H_{\fm}^{d} (R)$.  Therefore, we get the exact sequence 
$ 0 \rightarrow \H_{\fm}^{d-1} (R/\omega^{[p]}) \rightarrow \E \rightarrow \H_{\fm}^d (R) \rightarrow  0.$ 

Apply $\D_R(-):=\Hom_R(-, \E)$ gives the exact sequence
$ 0 \rightarrow \omega \xrightarrow{\beta} R \rightarrow \D_R(\H_{\fm} ^{d-1} (R/{\omega ^{[p]}})) \rightarrow 0.$ Therefore \begin{center}
$\omega \cong\im\beta$  and  $R/\im\beta\cong\D_R(\H_{\fm} ^{d-1} (R/{\omega ^{[p]}})).$ 
\end{center}
We show that $R/\im\beta$ is the canonical module of $R/\omega^{[p]}$. 
 As mentioned above, $R/ \omega^{[p]}$ admits canonical module $K_{R/ \omega^{[p]}}$, we have 
$\H_\fm^{d-1} (R/ \omega^{[p]})\cong\H_{{\fm} / \omega^{[p]}} ^{d-1} (R/ \omega^{[p]}) \cong \D' ( K_{R/ \omega^{[p]}} )$, where $\D'(-) = \Hom_{R/ \omega^{[p]}} (- , \E _{R / \omega^{[p]}}(R/{\fm}))$. On the other hand, one has  $\E _{R / \omega^{[p]}}(R/{\fm}) \cong \Hom _R(R/ \omega^{[p]}, E_R(R/{\fm})) $ as $R/ {\omega^{[p]}}$-module.
Using adjointness one can see that $\H_{\fm}  ^{d-1} (R/ \omega^{[p]}) \cong \Hom _R( K_{R/ \omega^{[p]}}, \E) = \D_R (K_{R/ \omega^{[p]}})$, as $R$-module. Hence $K_{R/{\omega ^{[p]}}} \cong R/\im\beta$ as $R/\omega ^{[p]}$-module.

As $R/\omega^{[p]}$ is equidimensional and unmixed, 
$R/\im\beta $ is faithful  as $R/{\omega^{[p]}}$-module (see \cite[(1.8)]{Aoy83} or \cite[(2.2) Remark]{HH94} ).
Now we have $\im\beta / \omega^{[p]}\subseteq (0 :_{R/\omega^{[p]}} R/ \im\beta)= 0$ . Therefore $\omega \cong \im\beta = \omega^{[p]}$, as desired.
\end{proof}

\begin{cor}\label{M'}\emph{(Compare with \cite[Theorem 4.1 (b)$\Rightarrow$(c)]{MA14})}
Let $(R, \fm)$ be a Cohen-Macaulay local ring with canonical module $\omega _R$. If $R$ is generically Gorenstein and $\omega _R \cong \omega_R ^{[p]}$, then $R$ is $\F$PI. 
\end{cor}

\begin{cor}
Let $(R, \fm)$ be an $\emph\F$PI Cohen-Macaulay local ring. Suppose that $R$ admits a canonical module $\omega_R$ such that $R / {\omega_R}^{[p]}$ is unmixed. If $\Tor_1 ^R(^f\hspace{-0.5mm}R , R/\omega_R)=0$, in particular if $\F(\omega_R)$ is torsion-free, then $R$ is Gorenstein.	
\end{cor}	

\begin{proof}
Applying $\F$ to 
$ 0 \to \omega_R \to R \to R/ \omega_R \to 0$ implies the exact sequence $ 0 \to \F(\omega_R) \to R \to R / {\omega_R ^{[p]}} \to 0.$ Hence $\F(\omega_R) \cong \omega_R ^{[p]}\cong\omega_R$ by \ref{1}, which implies that $\omega_R$ is free by \cite[ Korollar 4.3]{HE74}.
\end{proof}

\section{The effect of  the Frobenius dual  functors on $\F$PI rings}
The dual of the Peskine-Szpiro functor $\F^e$,
which has been studied by Herzog in \cite{HE74}, is the functor $\widetilde{\F}_R ^e (-)
= \Hom_R ({^{f^e}\hspace{-0.5mm}R}, -) $ on the category of $R$-modules. As mentioned in the Preliminaries section, for an $R$-module $M$, $\widetilde{\F}^e(M)$ has an $R$-module structure via the right $R$-action on ${^{f^e}R}$. 

Throughout this section the ring $R$ is assumed to be $\F$-finite so that $R$ admits a canonical module $K_R$ ( when $R$ is Cohen-Macaulay we write $\omega_R$ instead). 

We first recall the following interesting result which gives a characterization for $R$ to be $\F$PI in terms of $\widetilde{\F}$.  

\begin{prop}\label{aa}\cite[Proposition 3.10]{MA14} 
If  $R$ is an $\emph\F$-finite local ring, then $R$ is $\emph\F$PI if and only if $\widetilde{\emph\F}(R)\cong R$.
\end{prop} 

We first apply this result to show that $R[X]$ inherits the property of being $\F$PI under mild conditions. 
\begin{prop}\label{Canonical}
	Let $R$ be an $\F$-finite ring which satisfies Serre condition $(S_2)$. If $R$ is $\F$PI then $R[X]$ is $\F$PI.
\end{prop}

\begin{proof}
Note that our conditions imply that $R$ admits a canonical ideal $K_R$, say. Let $\fP$ be a maximal ideal of $R[X]$ and set $\fp := \fP \cap R $. By \cite[Proposition 3.3 (b)]{MA14}, it is enough to show that $R[X]_\fP$ is $\F$PI. There is an isomorphism $R[X]_\fP \cong (R_{\fp}[X])_{\fQ}$, where $\fQ$ is a maximal ideal of $R_{\fp}[X]$ and $\fQ \cap R_{\fp} = \fp R_{\fp}$. Therefore one can assume that $(R, \fm)$ is a local ring and $\fP$ maximal ideal of $R[X]$ such that $\fP \cap R = \fm$. Thus, by \cite[Korollar 5.21]{HK71}, $R[X]_\fP$ admits a canonical module $L$, say, and 
	\begin{equation}\tag{\ref{Canonical}.1}
	L \cong K_R \otimes_R R[X]_\fP. 
	\end{equation}
	There is a faithfully flat homomorphism $R \rightarrow R[X]_\fP$ whose fibres are Gorenstein and so $R[X]_\fP$ is $(S_2)$ and generically Gorenstein. This implies that $L$ is a height one ideal by \cite[Proposition 2.4]{LMa14}. Now, we have $\widetilde{\F}_{R[X]_\fP} (R[X]_\fP) \cong \Hom_{R[X]_\fP}({L}^{[p]}, L)$ by \cite[Korollar 5.9]{HE74} (note that the same argument in the proof of Korollar 5.9 works without Cohen-Macaulayness on $R$). Hence from (\ref{Canonical}.1) we get the following series of isomorphisms
	\[\begin{array}{lllll}
		\Hom_{R[X]_\fP}((K_R \otimes_R R[X]_{\fP})^{[p]}, K_R \otimes_R R[X]_\fP) &\cong \Hom_{R[X]_\fP}(((K_R \otimes_R R[X])^{[p]})_{\fP}, (K_R \otimes_R R[X])_{\fP})\\
		&\cong (\Hom_{R[X]}({K_R}^{[p]}[X], K_R[X] ))_{\fP} \\
		&\cong (\Hom_R({K_R}^{[p]}, K_R)[X]) _{\fP}\\ &\cong (\widetilde{\F}(R)[X])_{\fP}.
	\end{array}\]
	As $R$ is $\F$PI, $\widetilde{\F}(R) \cong R$ by Proposition \ref{aa} and therefore, $\widetilde{\F}_{R[X]_\fP} (R[X]_\fP) \cong R[X]_\fP$. The result follows by Proposition \ref{aa}.
\end{proof}
We next deduce that the condition $R$ to be $\F$PI is equivalent to a functorial property. More precisely, the following theorem.   
  
\begin{thm} \label{thmFF}
Let $(R, \fm)$ be an $\F$-finite local ring. Then $R$ satisfies $\emph\F$PI if and only if  $\emph\F^e(-)^{*}$ and   $\widetilde{\emph\F}^e (-^{*})$ are equivalent functors on the category of finitely generated $R$-modules for any $e >0$.
\end{thm}

\begin{proof}
Assume that $R$ satisfies $\F$PI. By Proposition \ref{aa}, 
$\widetilde{\F}(R) \cong R$ and so $\widetilde{\F}^e(R)\overset{\tau}{\cong} R$ for any $e>0$. Let $M$ be a finitely generated $R$-module with a finite presentation $R^{m'} \overset{(a_{ij})}{\longrightarrow} R^m \rightarrow M \rightarrow 0$, where $(a_{ij})$ is an $m\times m'$ matrix with entries  in $\fm$. Applying $\F^e$ and dualizing by $(-)^{*}$, we obtain the exact sequence 
$0 \rightarrow \F^{e}(M)^{*} \rightarrow R^m \xrightarrow{(a_{ji}^{p^e})} R^{m'}$.
Now by dualizing the minimal presentation, we have the exact sequence $0 \rightarrow M^{*} \rightarrow R^m \xrightarrow{(a_{ji})} R^{m'}$. Applying $\widetilde{\F}^{e}$, we obtain an exact sequence $ 0 \rightarrow \widetilde{\F}^e (M^{*}) \rightarrow \widetilde{\F}^e (R^m) \xrightarrow{\widetilde{\F}^e (a_{ji})} \widetilde{\F}^e (R^{m'}) $ and $\tau^{m'}\widetilde{\F}^e(a_{ji})=(a_{ji}^{p^e})\tau^m$ (see the proof of \cite[Korollar 5.5]{HE74}). Hence there exists a map $\widetilde{\F}^e (M^{*}) \overset{\tau_M}{\longrightarrow} \F^e (M)^{*}$ making the diagram 
\[\begin{tikzpicture}[descr/.style={fill=white,inner sep=1.5pt}]
\matrix (m) [
matrix of math nodes,
row sep=2.5em,
column sep=4em,
text height=1.5ex, text depth=0.25ex
]
{0 & \widetilde{\F}^e (M^{*}) & \widetilde{\F}^e (R^m) & \widetilde{\F}^e (R^{m'})  \\ 0 & \F^e (M)^{*} & R^m & R^{m'} \\ };
\path[overlay,->, font=\scriptsize,>=latex]
(m-1-1) edge  (m-1-2)
(m-1-2) edge  (m-1-3)
(m-1-2)  edge [densely dotted] node[right]{$\tau_M$}  (m-2-2)
(m-1-3) edge node[above]{$\widetilde{\F}^e(a_{ji})$} (m-1-4)
(m-2-1) edge  (m-2-2)
(m-2-2) edge  (m-2-3)
(m-2-3) edge node[above]{$(a_{ji}^{p^e})$}(m-2-4)
(m-1-3) edge node[right]{$\tau ^m$} (m-2-3)
(m-1-4) edge node[right]{$\tau ^{m'}$} (m-2-4)
;\end{tikzpicture}\] commutative with isomorphisms $\tau^m$ and $\tau^{m'}$ so that $\tau_M$ is an isomorphism.  

Let $g : M \rightarrow N$ be an $R$-homomorphism of finitely generated $R$-modules and let  $\mu: R^m\twoheadrightarrow M$ and $\nu:R^n\twoheadrightarrow N$ be some epimorphisms. Then there exists an $n\times m$ matrix $(r_{ij}):R^m\rightarrow R^n$ such that $g\mu=\nu (r_{ij})$. 
Applying the functors $\F^e(-)^{*}$ and $\widetilde{\F}((-)^{*})$ implies the diagram

\[\begin{tikzpicture}[descr/.style={fill=white,inner sep=1.5pt}]
\matrix (m) [
matrix of math nodes,
row sep=1.25em,
column sep=1.25em,
text height=1.5ex, text depth=0.25ex
]
{\widetilde{\F}^e(R^n) &  \mbox{}&\mbox{} & \mbox{} & \widetilde{\F}^e(R^m) \\
\mbox{} & \widetilde{\F}^e(N^{*}) &\mbox{}& \widetilde{\F}^e(M^{*}) & \mbox{} \\
\mbox{} &\mbox{}&0&\mbox{}&\mbox{}\\
\mbox{}& \F ^e(N)^{*} &\mbox{}& \F ^e(M)^{*}  & \mbox{} \\
R^n & \mbox{} &\mbox{}& \mbox{} & R^m. \\ };
\path[overlay,->, font=\scriptsize,>=latex]
(m-1-1) edge (m-1-5)
(m-2-2) edge (m-2-4)
(m-2-4) edge (m-1-5)
(m-2-2) edge (m-4-2)
(m-2-4) edge (m-4-4)
(m-2-2) edge (m-1-1)
(m-4-2) edge (m-4-4)
(m-4-4) edge (m-5-5)
(m-3-3) edge[densely dotted] (m-4-4)
(m-5-1) edge node[above]{$(r_{ji}^{p^e})$} (m-5-5) 
(m-1-1) edge  node[right]{$\cong $}  (m-5-1)
(m-1-5) edge  node[right]{$\cong $}  (m-5-5)
(m-4-2) edge (m-5-1)
;
\end{tikzpicture}\]

All faces of the diagram, except the front one, are clearly commutative. As the map $\F^e(M)^{*} \rightarrow R^m$ is injective, it follows that the front face is also commutative.  Therefore $\F^e(-)^{*}$ and $\widetilde{\F}^e (-^{*})$ are equivalent functors. 

For the converse, we have  $R \cong \F(R)^{*} \cong \widetilde{\F} (R^{*}) \cong \widetilde{\F} (R)$. By Proposition \ref{aa}, $R$ is $\F$PI.
\end{proof}

We may conclude $\F^e(M)^{*} \cong \widetilde{\F}^e (M^{*})$ for any $R$-module which is not necessarily finitely generated. More precisely we state the following. 
\begin{cor} \label{cor3.5}
Let $(R, \fm)$ be an $\F$-finite $\emph\F$PI local ring. Then $\emph\F^e(M)^{*} \cong \widetilde{\emph\F}^e (M^{*})$, for any $R$-module $M$ and any $e>0$.
\end{cor}

\begin{proof}
One can write $M \cong\underset{i}{\varinjlim} M_i$, where $M_i$'s are finitely generated $R$-submodules of $M$. Now by the following series of isomorphisms we complete the proof:
\[\begin{array}{rlllllllllll}
\F^e (M)^{*} &\cong (\F^e (\underset{i}{\varinjlim} M_i))^* &\cong (\underset{i}{\varinjlim} \F^e(M_i))^*\ \\ 
&\cong\underset{i}{ \varprojlim} (\F^e(M_i)^*)\ & & \text{( by \cite[Proposition 5.26]{JR79})}\\ &\cong\underset{i}{ \varprojlim} \widetilde{\F}^e (M_i^*)\ && \text{(by Theorem \ref{thmFF})}\\  &\cong \widetilde{\F}^e( \underset{i}{ \varprojlim}(M_i^*))\ & & \text{( by \cite[Proposition 5.21]{JR79})} \\ 
&\cong  \widetilde{\F}^e( (\underset{i}{\varinjlim}M_i)^*)\ &\cong \widetilde{\F}^e (M^{*}).
\end{array}\]
\end{proof}
A local ring $R$, with a canonical module $K_R$, is called {\it quasi-Gorenstein} if $K_R \cong R$. It is easy to see that $R$ is Gorenstein if and only if $R$ is Cohen-Macaulay and quasi-Gorenstein. Every quasi-Gorenstein ring is $\F$PI (see \cite[Proposition 3.6]{MA14}). By applying Theorem \ref{thmFF}, we show that for a ring to be $\F$PI is equivalent to be quasi-Gorenstein under some conditions. 

Recall that an $R$-module $M$ for which the natural map $M\to\Hom_R(\Hom_R(M, R), R)$ is an isomorphism is called reflexive.

\begin{thm}\label{thm3}
Let $(R, \fm)$ be an $\F$-finite local ring which is $(S_2)$ and its canonical module $K_R$ is reflexive. Suppose that $K_R \otimes_R  \emph\F^{e_0}(K_R^{*})$ is $(S_2)$,  for some $e_0>0$. Then the following statements are equivalent. 
\begin{enumerate}[\em(i)]
\item $R$ is quasi-Gorenstein.
\item $R$ is $\emph\F$PI.
\item For any $R$-module $M$ and any $e >0$, we have $\emph\F^e(M)^{*} \cong \widetilde{\emph\F}^e (M^{*})$.
\item   $\emph\F^{e_0}(K_R^{*})^{*} \cong \widetilde{\emph\F}^{e_0} (K_R^{**})$.
\end{enumerate}
\end{thm} 

\begin{proof}
(i)$\Rightarrow$(ii)  \cite[Proposition 3.6]{MA14}. 
	
(ii)$\Rightarrow$(iii) is clear by Theorem \ref{thmFF}. 

(iii)$\Rightarrow$(iv) is trivial.
	
(iv)$\Rightarrow$(i). We have  $\F^{e_0}(K_R^{*})^{*} \cong \widetilde{\F}^{e_0} (K_R^{**}) \cong \widetilde{\F}^{e_0} (K_R)\cong K_R$ by \cite[Satz 5.12]{HK71}.  As $R$ is $(S_2)$, 
$\Hom_R(K_R, \F^{e_0}(K_R^{*})^{*} ) \cong \Hom_R (K_R, K_R))\cong R$. 
Using adjointness, we have $(K_R \otimes_R \F^{e_0} (K_R^{*}))^{*} \cong R $. As, by our assumption $K_R \otimes_R \F^{e_0} (K_R^{*})$ is  $(S_2)$, it follows that  $K_R \otimes_R \F^{e_0} (K_R^{*})$ is free (see \cite[Theorem 3.10]{DEL19}). Therefore $K_R$ is free, as desired.   
\end{proof}

\begin{prop} \label{propFPIFPR}
Assume that $R$ is an $\F$-finite  $\emph\F$PI ring. Then, for any reflexive $R$-module $M$, $\widetilde{\emph\F}(M)$ is reflexive.
\end{prop}

\begin{proof}
Note that $\widetilde{\F}(M)$ is reflexive if and only if $\widetilde{\F}(M_\fp)$ is reflexive $R_\fp$-module for all prime ideals $\fp$ of $R$. Thus, by \cite[Proposition 3.3 (c)]{MA14}, one may assume that $R$ is local.  
Suppose that $M$ is a reflexive $R$-module so that there exists an exact sequence of this form $0 \rightarrow M \rightarrow R^n \rightarrow R^m$. Applying $\widetilde{\F}$ and using the fact that $\widetilde{\F} (R) \cong R$ (see Proposition \ref{aa}), implies that $0 \rightarrow \widetilde{\F}(M) \rightarrow R^n \rightarrow R^m$. Note that $R$ is generically Gorenstein by Proposition \ref{m1} and so $\widetilde{\F}(M)$ is reflexive (see \cite[Proposition 4.21 (c)$\Rightarrow$(a)]{Aus69}).
\end{proof}
It is shown in \cite[Proposition 3.5]{MA14} that, for an $\F$PI ring $R$, if $I$ is an injective $R$-module then $\F(I) \cong I$. However, in such a ring, $\widetilde{\F}(M) \ncong M$ for some reflexive $R$-module $M$. To see this, consider a zero dimensional
$\F$PI ring $R$, so that $R$ is Gorenstein, and set $M$ to be a non-free $R$-module. If $\widetilde{\F}(M) \cong M$, by Proposition \ref{DFtildF}, we have $\F(\D_R(M)) \cong \D_R(M)$, where $\D_R(-) = \Hom_R(-,R)$. Iterating $\F$ and using Remark \ref{rem2.1} gives $\D_R(M)$ is free and so $M$ is free, a contradiction.

The above proposition motivates us to bring up the following definition. 

\begin{defn}
\em{An $\F$-finite ring $R$ is said to be $\widetilde{\F}$-{\it preserve reflexive}, denoted by $\widetilde{\F}$PR, if $\widetilde{\F} (M)$ is reflexive for any reflexive $R$-module $M$.}
\end{defn}

It is known that every $\F$PI ring is generically Gorenstein.
Here we show that the same is true for a ring satisfying $\widetilde{\F}$PR.
\begin{lem} \label{propFPR(S)}
Any $\widetilde{\emph\F}$PR ring is generically Gorenstein.
\end{lem}

\begin{proof}
Pick $\fp \in \Min R$.  By Proposition \ref{DFtildF}, $\D_{R_\fp} (\widetilde{\F}^e (R_{\fp})) \cong \F ^e (\E_{R_\fp} (k(\fp)))$ which is free $R_\fp$-module for sufficiently large $e$ because $\dim R_{\fp}=0$ (see Remark \ref{rem2.1}).  Thus $\widetilde{\F}^e (R_{\fp}) \cong \overset{\text{t}}\oplus \E_{R_\fp} (k(\fp))$ for some $t$. As $\widetilde{\F}^e(R)$ is reflexive, $\E_{R_\fp} (k(\fp))^{**} \cong \E_{R_\fp} (k(\fp))$ which implies that $\E_{R_\fp} (k(\fp)) \cong R_{\fp}$ and so $R_\fp$ is Gorenstein.
\end{proof}
\begin{cor} \label{GFtildF}
Let $R$ be an $\F$-finite Artinian local ring. The following statements are equivalent. 
\begin{enumerate}[\rm(i)]
\item $R$ is Gorenstein.
\item $R$ is $\emph\F$PI.
\item $R$ is $\widetilde{\emph\F}$PR.
\end{enumerate}
\end{cor}
\begin{proof}
(i)$\Rightarrow$(ii) is clear. 
(ii)$\Rightarrow$(i) is proved in Proposition \ref{m1}. (i)$\Leftrightarrow$(iii) is clear by \cite[Exercise 3.2.15]{BH98} and Lemma \ref{propFPR(S)}.
\end{proof}

The property of $\widetilde{\F}$PR is not well-behaved upon taking a quotient by a nonzerodivisor.
Assume that $R$ is a one-dimensional non-Gorenstein local   ring satisfying $\widetilde{\F}$PR  (e.g. see \cite[after Proposition 4.2]{MA14}). For a non-zero divisor $x$, if $R/xR$ is $\widetilde{\F}$PR, then $R$ is Gorenstein by Corollary \ref{GFtildF} which is not true. 

Assume that  $\depth_R(M)=0$. In the following we examine depths of  $\F(M)$ and ${^{f^e}\hspace{-0.5mm}M}$. 
First note that one has $\depth_R ({^{f^e}\hspace{-0.5mm}M}) = \depth_R(M) = 0 $ for all $R$-modules $M$ but there exists a local ring $R$ and an $R$-module $M$ such that $\depth_R( M) = 0<\depth_R (\F (M))$ (see \cite[Remark 2.1.7 ]{Mil03}).  In a private conversation with Eghbali, Dao has proposed a more
precise observation.
\begin{lem}\label{lemDao}
Let $(R, \fm, k)$ be an $\F$-finite local ring and $M$ a finitely generated $R$-module such that $\emph\depth_R(M) = 0$. Then $k$ is a direct summand of ${^{f^e}\hspace{-0.5mm}M}$, for $e\gg 0$.
\end{lem}
\begin{proof}
As $\depth_R(M) = 0 $, $\Gamma_{\fm} (M) \nsubseteq \fm ^{[q']} M $ for some $q' = p^{e'}$. Also $ \ell_R (\Gamma_{\fm} (M)) < \infty$ implies that ${^{f^{e''}}\hspace{-0.5mm}\Gamma_{\fm} (M)}$ is a $k$-vector space for $e'' \gg 0$. Set $e := e' + e''$ and $N := M/\Gamma _{\fm} (M)$. Applying the functor ${^{f^e}}\hspace{-0.5mm}(-)$ to $ 0 \rightarrow \Gamma_{\fm} (M) \rightarrow M \rightarrow N \rightarrow 0$, implies the exact sequence
\begin{equation}\tag{\ref{lemDao}.1}
0 \rightarrow \overset{\text{finite}}\oplus k \rightarrow {^{f^e}\hspace{-0.5mm}M} \rightarrow {^{f^e}\hspace{-0.5mm}N} \rightarrow 0.
\end{equation}
As $\fm ^{[q]} M\subsetneq\fm ^{[q]} M+\Gamma_{\fm} (M)  \subseteq  M $, $ \ell_R (M/{\fm^{[q]} {M}}) > \ell_R (N/{\fm^{[q]} {N}})$. Therefore, by Remark \ref{rem2.1}. one has the inequality $\mu_R ({^{f^e}\hspace{-0.5mm}M}) > \mu_R ({^{f^e}\hspace{-0.5mm}N})$ between the minimum numbers of generators. Hence from (\ref{lemDao}.1) there is a split injection $k\to {^{f^e}\hspace{-0.5mm}M}$.	
\end{proof}
As an application, we are able to recover the following result about the Gorenstein dimension, $\gd_R(-)$.
\begin{prop} \label{G-dim}\cite[Theorem 6.2]{TY04}
Let $(R, \fm , k)$ be an $\F$-finite local ring. The following conditions are equivalent.
\begin{enumerate}[\rm(i)]
\item R is Gorenstein.
\item $\emph\gd_R ({^{f^e}\hspace{-0.5mm}R})< \infty$ for every $e > 0$. 
\end{enumerate}
\end{prop}
\begin{proof}
(i)$\Rightarrow$(ii). Obvious by \cite[Theorem 4.20]{Aus69}.\\	
(ii)$\Rightarrow$(i). Set $t:= \depth R$. If $t = 0$, then by Lemma \ref{lemDao}, $k$ is a direct summand of ${^{f^e}\hspace{-0.5mm}R}$, for $e \gg 0$. Therefore, by \cite[Lemma (1.1.10) (c)]{CH00}, 
$\gd_R (k)< \infty $ and so $R$ is Gorenstein by \cite[Theorem 17]{VM00}. Suppose $t>0$ and that $\underline{x}=x_1,\cdots, x_t$ is a maximal $R$-regular sequence in $\fm$. For any $e$, we have the exact sequences $0 \rightarrow {^{f^e}\hspace{-0.5mm}(R/(x_1, \cdots, x_{i-1}))} \overset{x_i}\rightarrow {^{f^e}\hspace{-0.5mm}(R/(x_1, \cdots, x_{i-1}))} \rightarrow {^{f^e}\hspace{-0.5mm}(R/(x_1, \cdots, x_i))} \rightarrow 0,\ 1\leq i\leq t,$ from which we eventually have $\gd_R({^{f^e}\hspace{-0.5mm}(R/(x_1, \cdots, x_t))}) < \infty$ \cite[Theorem 18]{VM00}. But $\depth(R/(x_1, \cdots, x_t))= 0$, hence $k$ is a direct summand of $ {^{f^e}\hspace{-0.5mm}(R/(x_1, \cdots, x_t))}$ for $e \gg 0$. This implies that $\gd_R (k) < \infty$ and so $R$ is Gorenstein. 
\end{proof}
 
In Theorem \cite[Theorem 1.7]{PS73}, Peskine-Szpiro have shown that, for  a not necessarily $\F$-finite ring $R$, if $M$ is an $R$-module with $\pd_R(M)<\infty$ then  $\pd_R(\F(M))<\infty$. It is known that if $M$ has a finite injective dimension then $\id_R ( \widetilde{\F}(M)) < \infty$ (see \cite[Satz 5.2]{HE74}). The converse may not be true and an Artinian non-Gorenstein local ring with $\fm ^{[p]} = 0$ would be an example. But, in a one-dimensional local ring, the converse could be true under mild condition (see Theorem\ref{FtildeF} (a)). 

In the following, we investigate some results about rings for which $\widetilde{\F}(M)$ is injective for modules $M$.

\begin{prop} \label{lem2.10}
Let $(R, \fm)$ be a zero dimensional $\F$-finite local ring and $e>0$. Then $\widetilde{\emph\F}^e(M)$ is injective for any $R$-module $M$ if and only if $\fm^{[p^e]}=0$.
\end{prop}
	
\begin{proof}
For the ``only if" part, take $M = \D_R(R/{\fm})$. By Proposition \ref{DFtildF}, $\widetilde{\F}^e(\D_R(R/{\fm})) \cong \D_R(\F^e(R/{\fm}))$, hence $\F^e(R/{\fm}) \cong R^n$, for some $n>0$. Therefor $R/{\fm}^{[p^e]} \cong R$ and so $\fm^{[p^e]}= 0$. 		
For the converse, take a minimal presentation $R^m \xrightarrow{(a_{ij})} R^n \rightarrow \D_R(M) \rightarrow 0$ of $\D_R(M)$, where $a_{ij} \in \fm$. Applying $\F^e$, we have $R^m \xrightarrow{(a_{ij}^{[p^e]})} R^n \rightarrow \F^e(\D_R(M)) \rightarrow 0$. But $(a_{ij}^{[p^e]}) = 0$, hence $\F^e(\D_R(M)) \cong R^n$. Therefore $\widetilde{\F}^e(M)$ is injective.
\end{proof}

The above result may raise the natural question. For which rings $R$, is $\widetilde{\F} (R_{\fq})$ injective as an $R_{\fq}$-module for any minimal prime ideal $\fq$ of $R$? (e.g. $R$ is generically Gorenstein.) In order to answer the question, we use the notion of the delta-invariant.
	
Assume that $R$ is a Cohen-Macaulay local ring which admits a canonical module and that $M$ is a finitely generated $R$-module. Therefore a {\it Cohen-Macaulay approximation} of $M$ exists which is a short exact sequence 
$
0\longrightarrow Y \longrightarrow X \overset{\varphi}{\longrightarrow} M \longrightarrow 0 
$
such that $X$ is a maximal Cohen-Macaulay $R$-module and $Y$ is a finitely generated $R$-module with finite injective dimension. We say that the Cohen-Macaulay approximation is {\it minimal} if each endomorphism $ \psi $ of $ X $ with $ \varphi\circ \psi=\varphi $ is an automorphism of $ X $. Note that a minimal Cohen-Macaulay approximation of $ M $ exists and is unique up to isomorphism (see \cite[Theorem 11.16]{LW12} and \cite[Corollary 2.4]{HS97}). If $ 0\longrightarrow Y \longrightarrow X \overset{\varphi}{\longrightarrow} M \longrightarrow 0 
$ is a minimal Cohen-Macaulay approximation of $ M $, then  the maximum rank of all free direct summands of $ X $ is called the delta-invariant of $M$ and denoted by $ \delta_R(M) $ (see \cite[Exercise 11.47]{LW12} and \cite[Proposition 1.3]{D92}).

The following result gives us an answer to the question.

\begin{prop}\label{g1}
Let $(R, \fm, k)$ be a Cohen-Macaulay $\F$-finite local ring and $\fp$ be a minimal prime ideal of $R$. Then $\widetilde{\F} (R_{\fp})$ is injective as $R_{\fp}$-module if and only if there exists an $R$-module $M$ such that $\delta_R(M)\neq0$ and $\widetilde{\F}(M_{\fp})$ is injective. 			 
\end{prop} 	
	
\begin{proof}
Note that $R$ is $\F$-finite so that it is a homomorphic image of a Gorenstein local ring and  admits a canonical module  $\omega_R$.  The ``only if " part is obvious by taking $M=R$.
		
For the converse, as $\delta_R(M)\neq0$, there exists an exact sequence 
$0\rightarrow Y \rightarrow R\oplus X  \rightarrow 	M \rightarrow 0$ such that $X$ is a maximal Cohen-Macaulay $R$-module  and that $\id_R(Y)< \infty $. Applying $ \widetilde{\F}$ gives the following exact sequence 
\begin{equation}\tag{\ref{g1}.1}
0\rightarrow \widetilde{\F} (Y) \rightarrow \widetilde{\F}(R)\oplus \widetilde{\F}(X)  \rightarrow 	\widetilde{\F}(M) \rightarrow 0
\end{equation} 
since $\id_R (Y) < \infty $, $ \Ext^1 _R ({^f\hspace{-0.5mm}R}, Y) = 0$ (see \cite[Satz 5.2]{HE74}). Localizing (\ref{g1}.1) at $\fp$ gives the exact sequence 
\begin{equation}\tag{\ref{g1}.2}
0\rightarrow \widetilde{\F} (Y_\fp) \rightarrow \widetilde{\F}(R_\fp)\oplus \widetilde{\F}(X_\fp)  \rightarrow 	\widetilde{\F}(M_\fp) \rightarrow 0.
\end{equation}
As $\id_{R_{\fp}}(Y_{\fp}) = 0$, one has $ \widetilde{\F} (Y_{\fp})  \cong \overset{\text{finite}}{\oplus} \widetilde{\F} (\E (k(\fp)) )\cong\overset{\text{finite}}{\oplus}\E (k(\fp)) $ by \cite[Lemma 4.1]{HE74}.   Therefore the exact sequence (\ref{g1}.2) splits and so
$\widetilde{\F}(R_\fp)\oplus \widetilde {\F}(X_\fp) 
\cong \widetilde{\F} (Y_\fp) \oplus 	\widetilde{\F}(M_\fp).$
As the right hand side is an injective $R_{\fp}$-module,  so is $\widetilde{\F} (R_{\fp}) $.
\end{proof}
	 Inspired by Proposition 
\ref{G-dim}, we bring up the following result which gives an explicit answer to the questions whether $\widetilde{\F}$PR implies Gorensteinness or ${\F}$PI in the
one dimensional case.

\begin{thm} \label{FtildeF}
Let $(R, \fm)$ be a one-dimensional local ring. The following statements hold true.
\begin{enumerate} [\rm(a)]
\item $R$ is Gorenstein if and only if 
$R$ is $\widetilde{\emph\F}$PR and $\emph\id_R \widetilde{\emph\F}(R) < \infty$.
\item $R$ is $\emph\F$PI if and only if $R$ is Cohen-Macaulay and  $\widetilde{\emph\F}$PR, and $\widetilde{\emph\F}(R)$ has a non-trivial free direct summand.
\end{enumerate}
\end{thm}

\begin{proof}
(a)  Assume that $R$ is Gorenstein. By \cite[Proposition 1.5]{HS93} and  Proposition \ref{propFPIFPR}, $R$ is  $\widetilde{\F}$PR. By \cite[Satz 5.12]{HK71},  $\widetilde{\F}(R) \cong R$.
For the converse,  as $\widetilde{\emph\F}(R)$ is a finitely generated $R$-module of finite injective dimension, $R$ is Cohen-Macaulay. Note that the depth of the module $\Hom_R(^f\hspace{-0.5mm}R, R)$ is independent of the two possible $R$-module structures on it, and so $\widetilde{\F}(R)$ is a maximal Cohen-Macaulay $R$-module. 
Therefore, by \cite[3.3.28]{BH98}, $\widetilde{\F} (R) \cong \oplus \omega_R$. As $\widetilde{\F} (R)$ is reflexive, $\omega_R$ is reflexive. Hence $R$ is Gorenstein (see \cite[(3.2)]{AG85}).
	
(b) Assume that $R$ is $\F$PI. The claim is clear by \cite[Proposition 3.12]{MA14}, Proposition \ref{propFPIFPR} and Proposition \ref{aa}.
Assuming the converse, $R$ is generically Gorenstein by Lemma \ref{propFPR(S)}. As rank of $\widetilde{\F}(R)$ is one, we have $\widetilde{\F}(R) \cong R \oplus X$, where $X$ has no free summand with $\ell_R (X) < \infty$. By \cite[Korollar 5.9]{HE74}, we have $\widetilde{\F}(R) \cong \Hom_R(\omega_R^{[p]}, \omega_R )$. Applying $\Hom_R(- , \omega_R)$, gives $\omega_R \oplus \Hom_R (X, \omega_R) \cong \Hom_R(\Hom_R(\omega_R ^{[p]}, \omega_R), \omega_R)\cong \omega_R^{[p]}$ because  $\omega_R^{[p]}$ is a maximal Cohen-Macaulay $R$-module. As $\Hom_R (X, \omega_R) = 0$, we get $\omega_R \cong \omega_R^{[p]}$. The result follows by \cite[Theorem 4.1]{MA14}  or Corollary \ref{M'}.
\end{proof} 

\section*{Acknowledgments}

The authors would like to appreciate the referee for her/his comments which made many improvements. Thanks to Arash Sadeghi and Ehsan Tavanfar for their invaluable discussions. The second author is deeply grateful to Hailong Dao and Justin Lyle for their insights during his visit at the Department of Mathematics in the University of Kansas (2017).




\end{document}